\documentclass{amsart}

\usepackage{xcolor}

\usepackage{latexsym,amsmath,amssymb,epic}
\usepackage{latexsym,amssymb,epic}
\usepackage{amsmath,amsthm}
\usepackage{color}

\usepackage{soul}

\usepackage{hyperref}
\hypersetup{
	colorlinks=true, %set true if you want colored links
	linktoc=all, %set to all if you want both sections and subsections linked
	linkcolor=blue} %choose some color if you want links to stand out

\newtheorem{theorem}{Theorem}[section]
\newtheorem{lemma}[theorem]{Lemma}

\newtheorem{conjecture}[theorem]{Conjecture}

\allowdisplaybreaks

\begin{document}
	
\title[A Conjecture of Saikia on $t$--Colored Overpartitions]{An Elementary Proof of a Conjecture of Saikia on Congruences for $t$--Colored Overpartitions}

\author{James A. Sellers}
\address{Department of Mathematics and Statistics, University of Minnesota Duluth, Duluth, MN 55812, USA}
\email{jsellers@d.umn.edu}

\subjclass[2010]{11P83, 05A17}
	
\keywords{partitions, congruences, designated summands, generating functions, dissections}
	
\maketitle
\begin{abstract}
The starting point for this work is the family of functions 
$\overline{p}_{-t}(n)$ which counts the number of $t$--colored overpartitions of $n.$   
In recent years, several infinite families of congruences satisfied by $\overline{p}_{-t}(n)$ for specific values of $t\geq 1$ have been proven.    In particular, in his 2023 work,  Saikia proved a number of congruence properties modulo powers of 2 for $\overline{p}_{-t}(n)$ for $t=5,7,11,13$.  He also included the following conjecture in that paper:  
\newline 
\ 
%\newline 
\noindent 
Conjecture: 
For all $n\geq 0$ and primes $t$, we have 
\begin{eqnarray*}
\overline{p}_{-t}(8n+1) &\equiv & 0 \pmod{2}, \\
\overline{p}_{-t}(8n+2) &\equiv & 0 \pmod{4}, \\
\overline{p}_{-t}(8n+3) &\equiv & 0 \pmod{8}, \\
\overline{p}_{-t}(8n+4) &\equiv & 0 \pmod{2}, \\
\overline{p}_{-t}(8n+5) &\equiv & 0 \pmod{8}, \\
\overline{p}_{-t}(8n+6) &\equiv & 0 \pmod{8}, \\
\overline{p}_{-t}(8n+7) &\equiv & 0 \pmod{32}.
\end{eqnarray*}
Using a truly elementary approach, relying on classical generating function manipulations and dissections, as well as proof by induction, we show that Saikia's conjecture holds for {\bf all} odd integers $t$ (not necessarily prime).

\end{abstract}

\section{Introduction}  

A {\it partition} of a positive integer $n$ is a finite non--increasing sequence of positive integers $\lambda_1 \geq \lambda_2 \geq \dots \geq \lambda_k$ such that $\lambda_1 + \lambda_2 + \dots + \lambda_k = n.$  We refer to the integers $\lambda_1, \lambda_2, \dots, \lambda_k$ as the {\it parts} of the partition.  For example, the number of partitions of the integer $n=4$ is 5, and the partitions counted in that instance are as follows:  
$$4, \ \ \ 3+1, \ \ \ 2+2, \ \ \ 2+1+1, \ \ \ 1+1+1+1$$
An {\it overpartition} of a positive integer $n$ is a partition of $n$ wherein the first occurrence of a part may be overlined.  
For example, the number of overpartitions of $n=4$ is 14 given the following list of overpartitions of 4:  
$$
4, \ \ \ \overline{4}, \ \ \ 3+1, \ \ \ \overline{3}+1, \ \ \ 3+\overline{1}, \ \ \ \overline{3}+\overline{1}, 
$$
$$
2+2, \ \ \  \overline{2} + 2, \ \ \ 2+1+1, \ \ \  \overline{2}+1+1, \ \ \  2+\overline{1}+1, \ \ \  \overline{2}+\overline{1}+1, 
$$
$$
1+1+1+1, \ \ \ \overline{1}+1+1+1
$$
The number of overpartitions of $n$ is often denoted $\overline{p}(n),$ so from the example above we see that $\overline{p}(4) = 14.$  
As noted by Corteel and Lovejoy \cite{CL}, the generating function for $\overline{p}(n)$ is given by 
\begin{equation*}
\sum_{n\geq 0} \overline{p}(n)q^n = \frac{f_2}{f_1^{2}}
\end{equation*}
where $f_r = (1-q^r)(1-q^{2r})(1-q^{3r})\dots$ for any positive integer $r.$   Several authors have proven congruences modulo powers of 2 satisfied by $\overline{p}(n)$; see \cite{CHSZ, HS_INT, HS_JCMCC, BK_INT, BK_DM, KM} for examples of such work.  

As a natural generalization, we define $\overline{p}_{-t}(n)$ to be the number of $t$--colored overpartitions of $n.$   (Note that some authors denote this family of functions as $\overline{p}_{t}(n).$)  
Saikia \cite{MS} notes that the generating function for $\overline{p}_{-t}(n)$ is given by 
\begin{equation}
\label{originalgenfn}
\sum_{n\geq 0} \overline{p}_{-t}(n)q^n = \frac{f_2^t}{f_1^{2t}} = \left( \frac{f_2}{f_1^{2}} \right)^t.
\end{equation}
In recent years, infinite families of congruences satisfied by $\overline{p}_{-t}(n)$ for specific values of $t>1$ have been proven; as examples, see \cite{KSV, BK_DM2, NN, MS, LW}.  In his recent work, Saikia \cite{MS} also included the following conjecture (in addition to proving a number of congruence properties modulo powers of 2 for $\overline{p}_{-t}(n)$ for $t=5,7,11,13$).

\begin{conjecture}
(Saikia, \cite{MS})  For all $n\geq 0$ and primes $t$, we have 
\begin{eqnarray*}
\overline{p}_{-t}(8n+1) &\equiv & 0 \pmod{2}, \\
\overline{p}_{-t}(8n+2) &\equiv & 0 \pmod{4}, \\
\overline{p}_{-t}(8n+3) &\equiv & 0 \pmod{8}, \\
\overline{p}_{-t}(8n+4) &\equiv & 0 \pmod{2}, \\
\overline{p}_{-t}(8n+5) &\equiv & 0 \pmod{8}, \\
\overline{p}_{-t}(8n+6) &\equiv & 0 \pmod{8}, \\
\overline{p}_{-t}(8n+7) &\equiv & 0 \pmod{32}.
\end{eqnarray*}
\end{conjecture}

Our primary goal in this note is to prove the following result which generalizes Saikia's conjecture mentioned above:  

\begin{theorem}
\label{mainthm}
For all $n\geq 0$ and all odd integers $t\geq 1,$ we have 
\begin{eqnarray*}
\overline{p}_{-t}(8n+1) &\equiv & 0 \pmod{2}, \\
\overline{p}_{-t}(8n+2) &\equiv & 0 \pmod{4}, \\
\overline{p}_{-t}(8n+3) &\equiv & 0 \pmod{8}, \\
\overline{p}_{-t}(8n+4) &\equiv & 0 \pmod{2}, \\
\overline{p}_{-t}(8n+5) &\equiv & 0 \pmod{8}, \\
\overline{p}_{-t}(8n+6) &\equiv & 0 \pmod{8}, \\
\overline{p}_{-t}(8n+7) &\equiv & 0 \pmod{32}.
\end{eqnarray*}
\end{theorem}

In contrast to the proof approach used by Saikia (which relies heavily on modular forms and the implementation of Smoot \cite{NS} of the algorithm of Radu \cite{SR} for proving such results), our proof of Theorem \ref{mainthm} is truly elementary, utilizing classical $q$--series results and generating function manipulations.  Indeed, the key to our proof revolves around a different way to view the generating function for $\overline{p}_{-t}(n)$ along with properties of expanding polynomials and the impact of the binomial coefficients and powers of 2 that arise.   

In order to complete the proof of Theorem \ref{mainthm}, we require three well--known lemmas which involve Ramanujan's functions 
$$\phi(q):=\sum_{k=-\infty}^\infty q^{k^2}   \textrm{\ \ \ \  and\ \ \ \ } \psi(q):=\sum_{k=0}^\infty q^{k(k+1)/2}.$$

\begin{lemma}
\label{2-dissection}
We have 
$$\phi(q) = \phi(q^4) + 2q\psi(q^8).$$  
\end{lemma}

\begin{proof}
See Hirschhorn \cite[(1.9.4)]{H}.  
\end{proof}

\begin{lemma}
\label{phi_neg_q}
We have 
$$
{\phi(-q)} = \frac{f_1^2}{f_2}.  
$$
\end{lemma}

\begin{proof}
See Hirschhorn \cite[(1.5.8)]{H}.  
\end{proof}

\begin{lemma}
\label{infprod}
We have 
$$
\frac{1}{\phi(-q)} = \prod_{i\geq 0} \phi(q^{2^i})^{2^i}.
$$
\end{lemma}

\begin{proof}
See Hirschhorn \cite[(1.5.16)]{H}.  
\end{proof}

%%%%%%%%%%%%%%%%%%%%%%%%%%%%%%%%%%%

\section{Proof of Theorem \ref{mainthm}}

We begin our proof of Theorem \ref{mainthm} by making the following key observation about the generating function for $\overline{p}_{-t}(n)$:  

\begin{theorem} 
\label{altgenfn}
For all $t\geq 1,$ 
$$
\sum_{n\geq 0} \overline{p}_{-t}(n)q^n = \left(  \prod_{i\geq 0} \phi(q^{2^i})^{2^i} \right)^t.
$$
\end{theorem}

\begin{proof}
This result follows immediately from (\ref{originalgenfn}) as well as Lemmas \ref{phi_neg_q} and \ref{infprod}.  
\end{proof}

Thus, we see that 
$$
\sum_{n\geq 0} \overline{p}_{-t}(n)q^n = \left(\phi(q)\phi(q^2)^2\phi(q^4)^4\right)^t \left(  \prod_{i\geq 3} \phi(q^{2^i})^{2^i} \right)^t
$$
where we have chosen to split off only the first three terms of the infinite product in Theorem \ref{altgenfn} because, in essence, Theorem \ref{mainthm} is a statement about the 8--dissection of the generating function for $\overline{p}_{-t}(n)$.  And since 
$$
\left(  \prod_{i\geq 3} \phi(q^{2^i})^{2^i} \right)^t
$$
is a function of $q^8$, it's role is not all that significant for the remainder of this work.  With this as our backdrop, we now proceed to prove the following theorem which will imply Theorem \ref{mainthm} (and handle all of the cases of that theorem simultaneously).  

\begin{theorem}
\label{mainthm2} 
For all odd $t\geq 1,$ 
$$
\left(\phi(q)\phi(q^2)^2\phi(q^4)^4  \right)^t = \left( \sum_{j=0}^7 a_{t,j}q^jF_{t,j}(q^8) \right)  
$$
\noindent where $F_{t,j}(q^8)$ is a function of $q^8$ whose power series representation has integer coefficients, and the following divisibilities hold:  
\begin{eqnarray*}
a_{t, 1} &\equiv & 0 \pmod{2}, \\
a_{t, 2} &\equiv & 0 \pmod{4}, \\
a_{t, 3} &\equiv & 0 \pmod{8}, \\
a_{t, 4} &\equiv & 0 \pmod{2}, \\
a_{t, 5} &\equiv & 0 \pmod{8}, \\
a_{t, 6} &\equiv & 0 \pmod{8}, \\
a_{t, 7} &\equiv & 0 \pmod{32}.
\end{eqnarray*}
\end{theorem}
 
\begin{proof} 
 We prove Theorem \ref{mainthm2} by induction on $t$.  The basis case of our proof by induction is the case $t=1,$ and this means that we wish to focus our attention on 
$
\left(\phi(q)\phi(q^2)^2\phi(q^4)^4\right)^1 =  \phi(q)\phi(q^2)^2\phi(q^4)^4.
$
Thanks to multiple applications of the 2--dissection result in Lemma \ref{2-dissection}, we know
\begin{eqnarray*}
\phi(q)\phi(q^2)^2\phi(q^4)^4 
&=& 
\left(\phi(q^4)+2q\psi(q^8)\right)\left(\phi(q^8)+2q^2\psi(q^{16}) \right)^2  \\
&& \ \ \ \ \ \ \ \ \times \left( \phi(q^{16}) + 2q^4\psi(q^{32})\right)^4 \\  
&=& 
\left(\left[\phi(q^{16}) + 2q^4\psi(q^{32}) \right]+2q\psi(q^8)\right)\left(\phi(q^8)+2q^2\psi(q^{16}) \right)^2  \\
&& \ \ \ \ \ \ \ \ \times  \left( \phi(q^{16}) + 2q^4\psi(q^{32})\right)^4.
\end{eqnarray*}
Using a computer algebra system, we see that the expanded version of this expression for $\phi(q)\phi(q^2)^2\phi(q^4)^4$ is the following:  

\begin{eqnarray*}
&& 128{q}^{24} \left( \psi \left( {q}^{16} \right)  \right) ^{2} \left( {
\psi} \left( {q}^{32} \right)  \right) ^{5} \\
&& {}+128{q}^{22}\phi \left( {q}^{8} \right) \psi \left( {q}^{16} \right)  \left( {
\psi} \left( {q}^{32} \right)  \right) ^{5} \\
&& {}+128{q}^{21}\psi \left( {q}^{8} \right)  \left( \psi \left( {q}^{16} \right) 
 \right) ^{2} \left( \psi \left( {q}^{32} \right)  \right) ^{4} \\
&& {}+ {q}^{20}\left( 32\, \left( \phi \left( {q}^{8} \right)  \right) ^
{2} \left( \psi \left( {q}^{32} \right)  \right) ^{5}+320\,{\phi} \left( {q}^{16} \right)  \left( \psi \left( {q}^{16}
 \right)  \right) ^{2} \left( \psi \left( {q}^{32} \right) 
 \right) ^{4} \right)  \\
&&{}+128{q}^{19}\phi \left( {q}^{8} \right) {
\psi} \left( {q}^{8} \right) \psi \left( {q}^{16} \right) 
 \left( \psi \left( {q}^{32} \right)  \right) ^{4} \\
&& {}+320{q}^{18}{\phi} \left( {q}^{8} \right) \phi \left( {q}^{16} \right) {
\psi} \left( {q}^{16} \right)  \left( \psi \left( {q}^{32}
 \right)  \right) ^{4} \\
&& {}+ {q}^{17}\left( 32\, \left( \phi \left( {q
}^{8} \right)  \right) ^{2}\psi \left( {q}^{8} \right)  \left( {
\psi} \left( {q}^{32} \right)  \right) ^{4} \right. \\
&&\ \ \ \ \ \ \ \ \ \ \ \left. +256\,\phi \left( {
q}^{16} \right) \psi \left( {q}^{8} \right)  \left( \psi
 \left( {q}^{16} \right)  \right) ^{2} \left( \psi \left( {q}^{32
} \right)  \right) ^{3} \right)  \\
&& {}+ {q}^{16}  \left( 80\, \left( \phi
 \left( {q}^{8} \right)  \right) ^{2}\phi \left( {q}^{16}
 \right)  \left( \psi \left( {q}^{32} \right)  \right) ^{4} \right. \\
&&\ \ \ \ \ \ \ \ \ \ \ \left. +320\,
 \left( \phi \left( {q}^{16} \right)  \right) ^{2} \left( {\psi} \left( {q}^{16} \right)  \right) ^{2} \left( \psi \left( {q}
^{32} \right)  \right) ^{3} \right) \\
&& {}+256{q}^{15} \phi \left( {q}
^{8} \right) \phi \left( {q}^{16} \right) \psi \left( {q}^{8
} \right) \psi \left( {q}^{16} \right)  \left( \psi \left( {
q}^{32} \right)  \right) ^{3}\\
&& {}+320{q}^{14}\phi \left( {q}^{8}
 \right)  \left( \phi \left( {q}^{16} \right)  \right) ^{2}{\psi} \left( {q}^{16} \right)  \left( \psi \left( {q}^{32}
 \right)  \right) ^{3} \\
&& {}+ {q}^{13} \left( 64\, \left( \phi \left( {q
}^{8} \right)  \right) ^{2}\phi \left( {q}^{16} \right) \psi
 \left( {q}^{8} \right)  \left( \psi \left( {q}^{32} \right) 
 \right) ^{3} \right. \\
&&\ \ \ \ \ \ \ \ \ \ \ \left. +192\, \left( \phi \left( {q}^{16} \right)  \right) 
^{2}\psi \left( {q}^{8} \right)  \left( \psi \left( {q}^{16}
 \right)  \right) ^{2} \left( \psi \left( {q}^{32} \right) 
 \right) ^{2} \right) \\
&&  {}+ {q}^{12} \left( 80\, \left( \phi \left( {q
}^{8} \right)  \right) ^{2} \left( \phi \left( {q}^{16} \right) 
 \right) ^{2} \left( \psi \left( {q}^{32} \right)  \right) ^{3} \right. \\
&&\ \ \ \ \ \ \ \ \ \ \ \left. +
160\, \left( \phi \left( {q}^{16} \right)  \right) ^{3} \left( {
\psi} \left( {q}^{16} \right)  \right) ^{2} \left( \psi
 \left( {q}^{32} \right)  \right) ^{2} \right) \\
&& {}+192{q}^{11} \phi
 \left( {q}^{8} \right)  \left( \phi \left( {q}^{16} \right) 
 \right) ^{2}\psi \left( {q}^{8} \right) \psi \left( {q}^{16
} \right)  \left( \psi \left( {q}^{32} \right)  \right) ^{2}\\
&& {}+160{q}^{10} \phi \left( {q}^{8} \right)  \left( \phi \left( {q}
^{16} \right)  \right) ^{3}\psi \left( {q}^{16} \right)  \left( {
\psi} \left( {q}^{32} \right)  \right) ^{2}\\
&& {}+ {q}^{9}\left( 48\,
 \left( \phi \left( {q}^{8} \right)  \right) ^{2} \left( {\phi} \left( {q}^{16} \right)  \right) ^{2}\psi \left( {q}^{8}
 \right)  \left( \psi \left( {q}^{32} \right)  \right) ^{2} \right. \\
&&\ \ \ \ \ \ \ \ \ \ \ \left. +64\,
 \left( \phi \left( {q}^{16} \right)  \right) ^{3}\psi
 \left( {q}^{8} \right)  \left( \psi \left( {q}^{16} \right) 
 \right) ^{2}\psi \left( {q}^{32} \right)  \right)  \\
&& {}+ {q}^{8}  \left( 40\, \left( \phi \left( {q}^{8} \right)  \right) ^{2}
 \left( \phi \left( {q}^{16} \right)  \right) ^{3} \left( {\psi} \left( {q}^{32} \right)  \right) ^{2} \right. \\
&&\ \ \ \ \ \ \ \ \ \ \ \left. +40\, \left( \phi
 \left( {q}^{16} \right)  \right) ^{4} \left( \psi \left( {q}^{16
} \right)  \right) ^{2}\psi \left( {q}^{32} \right)  \right) \\
&& {}+64{q}^{7}\phi \left( {q}^{8} \right)  \left( \phi \left( {q}^
{16} \right)  \right) ^{3}\psi \left( {q}^{8} \right) \psi
 \left( {q}^{16} \right) \psi \left( {q}^{32} \right)  \\
&& {}+40{q}^{6}\phi \left( {q}^{8} \right)  \left( \phi \left( {q}^{16}
 \right)  \right) ^{4}\psi \left( {q}^{16} \right) \psi
 \left( {q}^{32} \right)  \\
&& 
{}+ {q}^{5} \left( 16\, \left( \phi \left( 
{q}^{8} \right)  \right) ^{2} \left( \phi \left( {q}^{16}
 \right)  \right) ^{3}\psi \left( {q}^{8} \right) \psi
 \left( {q}^{32} \right) +8\, \left( \phi \left( {q}^{16}
 \right)  \right) ^{4}\psi \left( {q}^{8} \right)  \left( {\psi} \left( {q}^{16} \right)  \right) ^{2} \right)  \\
&& {}+ {q}^{4}\left( 10
\, \left( \phi \left( {q}^{8} \right)  \right) ^{2} \left( {\phi} \left( {q}^{16} \right)  \right) ^{4}\psi \left( {q}^{32}
 \right) +4\, \left( \phi \left( {q}^{16} \right)  \right) ^{5}
 \left( \psi \left( {q}^{16} \right)  \right) ^{2} \right)  \\
&& {}+8{q}^{3} \phi \left( {q}^{8} \right)  \left( \phi \left( {q}^{16
} \right)  \right) ^{4}\psi \left( {q}^{8} \right) \psi
 \left( {q}^{16} \right) \\
&& {}+4{q}^{2} \phi \left( {q}^{8} \right) 
 \left( \phi \left( {q}^{16} \right)  \right) ^{5}\psi
 \left( {q}^{16} \right) \\
&& {}+2q \left( \phi \left( {q}^{8}
 \right)  \right) ^{2} \left( \phi \left( {q}^{16} \right) 
 \right) ^{4}\psi \left( {q}^{8} \right)  \\
&& {}+ \left( \phi
 \left( {q}^{8} \right)  \right) ^{2} \left( \phi \left( {q}^{16}
 \right)  \right) ^{5}
\end{eqnarray*}
Note that, because all of the moduli in Theorem \ref{mainthm2} (i.e., $2, 4, 8,$ and $32$) are divisors of 32, it is sufficient to look at the expression for $\phi(q)\phi(q^2)^2\phi(q^4)^4$ (which appears above) modulo 32 and confirm that the divisibilities for the coefficients are satisfied.  Indeed, when we look at the above modulo 32, we obtain

\begin{eqnarray*}
&& 16{q}^{16} \,\phi \left( {q}^{16} \right)  \left( \phi \left( {q}^{8}
 \right)  \right) ^{2} \left( \psi \left( {q}^{32}
 \right)  \right) ^{4} \\
&& {}+16{q}^{12}\, \left( \phi \left( {q}^{16} \right) 
 \right) ^{2} \left( \phi \left( {q}^{8} \right)  \right) ^{2} \left( \psi \left( {q}^{32} \right)  \right) ^{3}  \\
&& {}+16{q}^{9} \,\psi \left( {q}^{8} \right)  \left( \phi \left( {q}^{8}
 \right)  \right) ^{2} \left( \phi \left( {q}^{16} \right) 
 \right) ^{2} \left( \psi \left( {q}^{32} \right)  \right) ^{2}  \\
&& {}+8{q}^{8}
 \left( \, \left( \phi \left( {q}^{16} \right)  \right) ^{3}
 \left( \phi \left( {q}^{8} \right)  \right) ^{2} \left( \psi \left( {q}^{32} \right)  \right) ^{2} +\, \left( \phi \left( {q
}^{16} \right)  \right) ^{4} \left( \psi \left( {q}^{16} \right) 
 \right) ^{2}\psi \left( {q}^{32} \right)  \right) \\ 
&&{}+8 {q}^{6} \,
 \left( \phi \left( {q}^{16} \right)  \right) ^{4}\phi
 \left( {q}^{8} \right) \psi \left( {q}^{16} \right) \psi
 \left( {q}^{32} \right)\\
&& {}+  8{q}^{5} \left( \,\psi \left( {q}^{8}
 \right)  \left( \psi \left( {q}^{16} \right)  \right) ^{2}
 \left( \phi \left( {q}^{16} \right)  \right) ^{4} +2\,\psi
 \left( {q}^{8} \right)  \left( \phi \left( {q}^{8} \right) 
 \right) ^{2} \left( \phi \left( {q}^{16} \right)  \right) ^{3}\psi \left( {q}^{32} \right)  \right) \\
&& {}+ 2{q}^{4}\left( 2\, \left( \phi \left( {q}^{16} \right)  \right) ^{5} \left( \psi
 \left( {q}^{16} \right)  \right) ^{2}  +5\, \left( \phi \left( {q
}^{16} \right)  \right) ^{4} \left( \phi \left( {q}^{8} \right) 
 \right) ^{2}\psi \left( {q}^{32} \right)  \right)  \\
&& {}+8{q}^{3}\,\psi \left( {q}^{8} \right) \phi \left( {q}^{8} \right) \psi \left( {q}^{16} \right)  \left( \phi \left( {q}^{16}
 \right)  \right) ^{4} \\
&& {}+4{q}^{2} \, \left( \phi \left( {q}^{16}
 \right)  \right) ^{5}\phi \left( {q}^{8} \right) \psi
 \left( {q}^{16} \right) \\
&& {}+2q\,\psi \left( {q}^{8} \right) 
 \left( \phi \left( {q}^{8} \right)  \right) ^{2} \left( \phi \left( {q}^{16} \right)  \right) ^{4} \\
&& {}+ \left( \phi \left( {q}^
{16} \right)  \right) ^{5} \left( \phi \left( {q}^{8} \right) 
 \right) ^{2}.
\end{eqnarray*}
Each of the numerical coefficients above satisfies the corresponding congruences from Theorem \ref{mainthm2}.  We summarize this divisibility information in the table below.  
\begin{small}
\begin{table}[h!]
%\centering
% \begin{tabular}{ |p{3cm}|p{3cm}|p{3cm}|  }
\begin{tabular}{ | c c c |  }
\hline
Term &  Power of $q$ Modulo 8 & Required Divisor  \\
\hline
$16q^{16}$ & 0 & None \\
$16q^{12}$ & 4 & 2 \\
$16q^9$ & 1 & 2  \\
$8q^8$ & 0 & None \\
$8q^6$ & 6 & 8 \\
$8q^5$ & 5 & 8 \\
$2q^4$ & 4 & 2 \\
$8q^3$ & 3 & 8 \\
$4q^2$ & 2 & 4 \\
$2q$ & 1 & 2 \\
\hline
\end{tabular}
\begin{center}
\caption{Divisibility Check for Basis Case \phantom{111111111111111}}
\end{center}
\label{table:1}
\end{table}
\end{small}

We next move to the induction step of our proof.  We assume that, for some odd $t\geq 1,$ 
$$
\left(\phi(q)\phi(q^2)^2\phi(q^4)^4  \right)^t = \left( \sum_{j=0}^7 a_{t,j}q^jF_{t,j}(q^8) \right) 
$$
where the appropriate divisibility properties are satisfied by the coefficients $a_{t,j}$  and the functions $F_{t,j}(q^8)$ have power series representation with integer coefficients (for $0\leq j \leq 7$).   We then wish to prove that 
$$
\left(\phi(q)\phi(q^2)^2\phi(q^4)^4  \right)^{t+2}  = \left( \sum_{j=0}^7 a_{t+2,j}q^jF_{t+2,j}(q^8) \right) 
$$
where the coefficients $a_{t+2,j}$ satisfy the appropriate divisibility properties and the functions $F_{t+2,j}(q^8)$ have power series representation with integer coefficients (for $0\leq j\leq 7$).  Note that 
\begin{eqnarray*}
\left(\phi(q)\phi(q^2)^2\phi(q^4)^4\right)^{t+2}  
&=&
\left(\phi(q)\phi(q^2)^2\phi(q^4)^4\right)^{t} \left(\phi(q)\phi(q^2)^2\phi(q^4)^4\right)^{2}   \\
&=& 
\left( \sum_{j=0}^7 a_{t,j}q^jF_{t,j}(q^8) \right)\left(\phi(q)\phi(q^2)^2\phi(q^4)^4\right)^{2}  
\end{eqnarray*}
where $a_{t,j}$ satisfies the appropriate divisibility properties as stated in Theorem \ref{mainthm2} thanks to the induction hypothesis.  

As was mentioned above, we can now perform our work modulo 32 and confirm the appropriate divisibility properties of the corresponding coefficients.  
With this in mind, we write 
\begin{eqnarray*}
\sum_{j=0}^7 a_{t,j}q^jF_{t,j}(q^8) 
&=& b_{t,0}F_{t,0}(q^8) + 2b_{t,1}qF_{t,1}(q^8) \\
\noalign{\vspace{-10pt}}
&& \ \ \ \ + 4b_{t,2}q^2F_{t,2}(q^8) + 8b_{t,3}q^3F_{t,3}(q^8) \\
&& \ \ \ \ + 2b_{t,4}q^4F_{t,4}(q^8) + 8b_{t,5}q^5F_{t,5}(q^8) \\
&& \ \ \ \ + 8b_{t,6}q^6F_{t,6}(q^8) + 32b_{t,7}q^7F_{t,7}(q^8)
\end{eqnarray*}
for some integers $b_{t,j},$ $j=1, 2, \dots, 7.$ Then we have (modulo 32)
\begin{eqnarray*}
&& \left( \sum_{j=0}^7 a_{t,j}q^jF_{t,j}(q^8) \right)\left(\phi(q)\phi(q^2)^2\phi(q^4)^4\right)^{2} \\
&\equiv & 
8{q}^{12} \, \left( \phi\left( {q}^{16} \right)  \right) ^{8} \left( \phi\left( {q}^{8} \right)  \right) ^{4} \left( \psi\left( 
{q}^{32} \right)  \right) ^{2}b_{t,4}F_{t,4} \\
&& \ \ \ \ 
{}+16{q}^{10}\, \left( \phi\left( {q}^{16} \right)  \right) ^{8} \left( \phi\left( {q}
^{8} \right)  \right) ^{4} \left( \psi\left( {q}^{32} \right)  \right) ^{2}b_{t,2}F_{t,2} \\
&& \ \ \ \ 
{}+ 8{q}^{9} \left( \, \left( \phi\left( {q}^{16} \right)  \right) ^{8} \left( \phi\left( {q}^{8}
 \right)  \right) ^{4} \left( \psi\left( {q}^{32} \right) 
 \right) ^{2}b_{t,1}F_{t,1} \right. \\
&& \ \ \ \ \ \ \ \ \ \ \ \ \ \left. 
{}+2\, \left( \phi\left( {q}^{8}
 \right)  \right) ^{4} \left( \phi\left( {q}^{16} \right) 
 \right) ^{8}\psi\left( {q}^{8} \right) \psi\left( {q}^{32} \right) b_{t,4}F_{t,4} \right)  \\
&& \ \ \ \ 
{}+ 4{q}^{8} \left( 4\, \left( \phi
 \left( {q}^{16} \right)  \right) ^{10} \left( \psi\left( {q}^{16} \right)  \right) ^{2} \left( \phi\left( {q}^{8} \right) 
 \right) ^{2}b_{t,4}F_{t,4} \right. \\
&& \ \ \ \ \ \ \ \ \ \ \ \ \ \left. 
{}+4\, \left( \phi\left( {q}^{16}
 \right)  \right) ^{10} \left( \psi\left( {q}^{16} \right) 
 \right) ^{4}b_{t,0}F_{t,0} \right. \\
&& \ \ \ \ \ \ \ \ \ \ \ \ \ \left. 
{}+5\, \left( \phi\left( {q}^{16}
 \right)  \right) ^{8} \left( \phi\left( {q}^{8} \right) 
 \right) ^{4} \left( \psi\left( {q}^{32} \right)  \right) ^{2}b_{t,0}F_{t,0} \right. \\
&& \ \ \ \ \ \ \ \ \ \ \ \ \ \left. 
{}+2\, \left( \phi\left( {q}^{16} \right)  \right) ^{9
} \left( \phi\left( {q}^{8} \right)  \right) ^{4}\psi
 \left( {q}^{32} \right) b_{t,4}F_{t,4} \right)  \\
&& \ \ \ \ 
{}+  8{q}^{6} \left( 2\,
 \left( \phi\left( {q}^{16} \right)  \right) ^{9} \left(\phi \left( {q}^{8} \right)  \right) ^{4}\psi\left( {q}^{32}
 \right) b_{t,2}F_{t,2} \right. \\
&& \ \ \ \ \ \ \ \ \ \ \ \ \ \left. 
{}+\, \left( \phi\left( {q}^{8} \right) 
 \right) ^{4} \left( \phi\left( {q}^{16} \right)  \right) ^{10}b
_{{t,6}}F_{t,6} \right.  \\
&& \ \ \ \ \ \ \ \ \ \ \ \ \ \left. 
{}+\, \left( \psi\left( {q}^{8} \right)  \right) ^{2} \left( \phi\left( {q}^{8} \right)  \right) ^{4} 
\left( \phi\left( {q}^{16} \right)  \right) ^{8}b_{t,4}F_{t,4} \right. \\
&& \ \ \ \ \ \ \ \ \ \ \ \ \ \left. 
{}+2\, \left( \phi\left( {q}^{16} \right)  \right) ^{10} \left( \phi
 \left( {q}^{8} \right)  \right) ^{3}\psi\left( {q}^{16}
 \right) b_{t,4}F_{t,4} \right. \\
&& \ \ \ \ \ \ \ \ \ \ \ \ \ \left. 
{}+2\, \left( \phi\left( {q}^{8} \right) 
 \right) ^{4} \left( \phi\left( {q}^{16} \right)  \right) ^{8}\psi\left( {q}^{8} \right) \psi\left( {q}^{32} \right) b_{t,1}F_{t,1} \right) \\
&& \ \ \ \ 
{}+ 8{q}^{5} \left( 2\, \left( \phi\left( {q}^{16} \right)  \right) ^{10} \left( \psi\left( {q}^{16} \right) 
 \right) ^{2} \left( \phi\left( {q}^{8} \right)  \right) ^{2}b_{t,1}F_{t,1} \right. \\
&& \ \ \ \ \ \ \ \ \ \ \ \ \ \left. 
{}+\, \left( \phi\left( {q}^{8} \right)  \right) ^{4}
 \left( \phi\left( {q}^{16} \right)  \right) ^{8}\psi\left( {q}^{8} \right) \psi\left( {q}^{32} \right) b_{t,0}F_{t,0} \right. \\
&& \ \ \ \ \ \ \ \ \ \ \ \ \ \left. 
{}+\, \left( \phi\left( {q}^{8} \right)  \right) ^{4} \left( \phi\left( {q}^{16} \right)  \right) ^{9}\psi\left( {q}^{8}
 \right) b_{t,4}F_{t,4} \right. \\
&& \ \ \ \ \ \ \ \ \ \ \ \ \ \left. 
{}+\, \left( \phi\left( {q}^{16} \right) 
 \right) ^{9} \left( \phi\left( {q}^{8} \right)  \right) ^{4}\psi\left( {q}^{32} \right) b_{t,1}F_{t,1} \right. \\
&& \ \ \ \ \ \ \ \ \ \ \ \ \ \left. 
{}+\, \left( \phi
 \left( {q}^{8} \right)  \right) ^{4} \left( \phi\left( {q}^{16}
 \right)  \right) ^{10}b_{t,5}F_{t,5} \right) \\
&& \ \ \ \ 
{}+ 2{q}^{4}\left( 8\,
 \left( \psi\left( {q}^{8} \right)  \right) ^{2} \left( \phi\left( {q}^{8} \right)  \right) ^{4} \left( \phi\left( {q}^{16} \right)  \right) ^{8}b_{t,2}F_{t,2} \right. \\
&& \ \ \ \ \ \ \ \ \ \ \ \ \ \left. 
{}+\, \left( \phi\left( {q}^{8} \right)  \right) ^{4} \left( \phi\left( {q}^{16} \right) 
 \right) ^{10}b_{t,4}F_{t,4} \right. \\
&& \ \ \ \ \ \ \ \ \ \ \ \ \ \left. 
{}+12\, \left( \phi\left( {q}^{16}
 \right)  \right) ^{10} \left( \psi\left( {q}^{16} \right) 
 \right) ^{2} \left( \phi\left( {q}^{8} \right)  \right) ^{2}b_{t,0}F_{t,0} \right. \\
&& \ \ \ \ \ \ \ \ \ \ \ \ \ \left. 
{}+10\, \left( \phi\left( {q}^{16} \right)  \right) ^{9} \left( \phi\left( {q}^{8} \right)  \right) ^{4}\psi
 \left( {q}^{32} \right) b_{t,0}F_{t,0} \right)  \\
&& \ \ \ \ 
{}+ 8{q}^{3} \left( \, \left( \phi\left( {q}^{8} \right)  \right) ^{4} \left( \phi\left( {q}^{16} \right)  \right) ^{10}b_{t,3}F_{t,3} \right. \\
&& \ \ \ \ \ \ \ \ \ \ \ \ \ \left. 
{}+\, \left( \psi\left( {q}^{8} \right)  \right) ^{2} \left( \phi\left( 
{q}^{8} \right)  \right) ^{4} \left( \phi\left( {q}^{16}
 \right)  \right) ^{8}b_{t,1}F_{t,1} \right. \\
&& \ \ \ \ \ \ \ \ \ \ \ \ \ \left. 
{}+2\, \left( \phi\left( {q}^{8} \right)  \right) ^{4} \left( \phi\left( {q}^{16} \right) 
 \right) ^{9}\psi\left( {q}^{8} \right) b_{t,2}F_{t,2} \right. \\
&& \ \ \ \ \ \ \ \ \ \ \ \ \ \left. 
{}+2\,
 \left( \phi\left( {q}^{16} \right)  \right) ^{10} \left( \phi\left( {q}^{8} \right)  \right) ^{3}\psi\left( {q}^{16}
 \right) b_{t,1}F_{t,1} \right) \\
&& \ \ \ \ 
{}+ 4{q}^{2} \left( \, \left(\phi
 \left( {q}^{8} \right)  \right) ^{4} \left( \phi\left( {q}^{16}
 \right)  \right) ^{10}b_{t,2}F_{t,2} \right. \\
&& \ \ \ \ \ \ \ \ \ \ \ \ \ \left. 
{}+2\, \left( \phi\left( {q}^{8} \right)  \right) ^{4} \left( \phi\left( {q}^{16} \right) 
 \right) ^{9}\psi\left( {q}^{8} \right) b_{t,1}F_{t,1} \right. \\
&& \ \ \ \ \ \ \ \ \ \ \ \ \ \left. 
{}+\,
 \left( \psi\left( {q}^{8} \right)  \right) ^{2} \left( \phi\left( {q}^{8} \right)  \right) ^{4} \left( \phi\left( {q}^{16} \right)  \right) ^{8}b_{t,0}F_{t,0} \right. \\
&& \ \ \ \ \ \ \ \ \ \ \ \ \ \left. 
{}+2\, 
\left( \phi\left( {q}^{16} \right)  \right) ^{10} \left( \phi\left( {q}^{8} \right) 
 \right) ^{3}\psi\left( {q}^{16} \right) b_{t,0}F_{t,0} \right)  \\
&& \ \ \ \ 
{}+ 2q\left( 2\, \left( \phi\left( {q}^{8} \right)  \right) ^
{4} \left( \phi\left( {q}^{16} \right)  \right) ^{9}\psi
 \left( {q}^{8} \right) b_{t,0}F_{t,0} \right. \\
&& \ \ \ \ \ \ \ \ \ \ \ \ \ \left. 
{}+\, \left( \phi\left( {q}
^{8} \right)  \right) ^{4} \left( \phi\left( {q}^{16} \right) 
 \right) ^{10}b_{t,1}F_{t,1} \right)  \\
&& \ \ \ \ 
{} + \left( \phi\left( {q}^{8
} \right)  \right) ^{4} \left( \phi\left( {q}^{16} \right) 
 \right) ^{10}b_{t,0}F_{t,0} \pmod{32}.
\end{eqnarray*}

Note that the coefficients of the corresponding powers of $q$ are divisible by the appropriate powers of 2.  We confirm this with the summarized information in the table below.  
\begin{small}
\begin{table}[h!]
\centering
\begin{tabular}{ | c c c |  }
\hline
Term &  Power of $q$ Modulo 8 & Required Divisor  \\
\hline
$8q^{12}$ & 4 & 2 \\
$16q^{10}$ & 2 & 4\\
$8q^9$ & 1 & 2  \\
$4q^8$ & 0 & None \\
$8q^6$ & 6 & 8 \\
$8q^5$ & 5 & 8 \\
$2q^4$ & 4 & 2 \\
$8q^3$ & 3 & 8 \\
$4q^2$ & 2 & 4 \\
$2q$ & 1 & 2 \\
\hline
\end{tabular}
\begin{center}
\caption{Divisibility Check for Induction Step \phantom{111111111111111}}
\end{center}
\label{table:2}
\end{table}
\end{small}

\noindent 
This gives the induction step, and the proof is complete.   
\end{proof}
\noindent 
We are now in a position to prove Theorem \ref{mainthm}.  

\begin{proof}
(of Theorem \ref{mainthm})  Using Theorems \ref{altgenfn} and  \ref{mainthm2}, we know that, for all odd $t\geq 1,$  
\begin{eqnarray*}
\sum_{n\geq 0} \overline{p}_{-t}(n)q^n 
&=& 
\left(  \prod_{i\geq 0} \phi(q^{2^i})^{2^i} \right)^t \\
&=& 
\left( \phi(q)\phi(q^2)^2\phi(q^4)^4\right)^t\left(  \prod_{i\geq 3} \phi(q^{2^i})^{2^i} \right)^t \\
&=& 
\left( \sum_{j=0}^7 a_{t,j}q^jF_{t,j}(q^8) \right)  \left(  \prod_{i\geq 3} \phi(q^{2^i})^{2^i} \right)^t 
\end{eqnarray*}
\noindent where $F_{t,j}(q^8)$ is a function of $q^8$ whose power series representation has integer coefficients, and the following divisibilities hold:  
\begin{eqnarray*}
a_{t, 1} &\equiv & 0 \pmod{2}, \\
a_{t, 2} &\equiv & 0 \pmod{4}, \\
a_{t, 3} &\equiv & 0 \pmod{8}, \\
a_{t, 4} &\equiv & 0 \pmod{2}, \\
a_{t, 5} &\equiv & 0 \pmod{8}, \\
a_{t, 6} &\equiv & 0 \pmod{8}, \\
a_{t, 7} &\equiv & 0 \pmod{32}.
\end{eqnarray*}
The result immediately follows.  
\end{proof}

%%%%%%%%%%%%%%%%%%%%%%%%%%%%%%%%%%% 

\section{Closing Thoughts}  
Several thoughts are in order as we close.  

\begin{itemize}

\item{} First, notice that we could look further at, say, 16--dissections or 32--dissections of the generating function for $\overline{p}_{-t}(n)$ (rather than stopping at 8--dissections as we did above) thanks to the beneficial structure of the generating function representation in Theorem \ref{altgenfn}.  However, it is clear that the mathematics will get more complicated.  We have chosen to stop at the 8--dissections as guided by Saikia's original conjecture.  

\vskip .25in

\item{} Secondly, we highlight that $t$ need not be prime (as was required in Saikia's original conjecture).  The results hold for any odd, positive integer $t.$    Related to this comment, note that no significant facts about the functions $\phi(q)$ or $\psi(q)$ come into play in the above results regarding divisibility by powers of 2, other than the appearance of the ``2'' in the 2--dissection in Lemma \ref{2-dissection}.  Indeed, it is the presence of the 2 in Lemma \ref{2-dissection} which serves as the ``engine'' for the divisibility properties in Theorem \ref{mainthm}.   Said a different way, these divisibility properties for $\overline{p}_{-t}(n)$ for odd $n$ have relatively little to do with the ``combinatorics'' of the colored overpartitions in question, and a great deal to do with the algebra involved in the manipulations of the generating function as described above.  

\vskip .25in

\item{} Note that there are no claims here to Theorem \ref{mainthm} being ``best possible''.  In particular, for specific values of $t$, the congruences in Theorem \ref{mainthm} can be improved.  For example, Saikia \cite[Theorem 1]{MS} proved that, for all $n\geq 0,$ $\overline{p}_{-5}(8n+7) \equiv 0 \pmod{128},$ an improvement of Theorem \ref{mainthm} where divisibility by only 32 is guaranteed for the arithmetic progression $8n+7$.  

\vskip .25in 

\item{} It is striking to note that there is no statement in Theorem \ref{mainthm} about what happens modulo a power of 2 for the values $\overline{p}_{-t}(8n+0)$ for $t$ odd.  From the perspective of the proof approach shared above, one might conclude that there is nothing that can be said about the divisibility modulo powers of 2 within the arithmetic progression $8n+0.$  However, a great deal can be said, if one takes a different viewpoint. As a start, consider the following ``global'' parity theorem which is easily proved:  

\begin{theorem}
\label{characterization_mod2}
For all $t\geq 1$ and all $n\geq 1,$ $\overline{p}_{-t}(n) \equiv 0 \pmod{2}.$  
\end{theorem}
\begin{proof}
From Theorem \ref{altgenfn} and the definition of $\phi(q),$ we know 
\begin{eqnarray*}
\sum_{n\geq 0} \overline{p}_{-t}(n)q^n 
&=& 
\left(  \prod_{i\geq 0} \phi(q^{2^i})^{2^i} \right)^t \\
&=& 
\left(  \prod_{i\geq 0} \left( 1+2\sum_{k=1}^\infty (q^{2^i})^{k^2} \right)^{2^i} \right)^t \\
&\equiv & 
\left(  \prod_{i\geq 0} (1+0)^{2^i} \right)^t  \pmod{2}\\
&=& 
1.
\end{eqnarray*}
The result follows.
\end{proof}
Thus, the only barrier to having a result included in Theorem \ref{mainthm} for the arithmetic progression $8n+0$ is the coefficient of $q^0$ since we define $\overline{p}_{-t}(0)=1.$    

As an aside, notice that we can take this ``global'' approach a step further and quickly prove a characterization modulo 4 satisfied by $\overline{p}_{-t}(n)$ for any odd $t\geq 1$.  
\begin{theorem}
\label{characterization_mod4}
For all odd $t\geq 1$ and all $n\geq 0,$   
$$\overline{p}_{-t}(n) \equiv  
\begin{cases}
  1  \pmod{4} & \text{if\ } n =0,  \\
  2 \pmod{4} & \text{if\ } n \text{ is a positive square}, \\
  0 \pmod{4} & \text{otherwise}.
\end{cases}
$$
\end{theorem}
\begin{proof}
Extending the logic from the proof of Theorem \ref{characterization_mod2}, we have 
\begin{eqnarray*}
\sum_{n\geq 0} \overline{p}_{-t}(n)q^n 
&=& 
\left(  \prod_{i\geq 0} \phi(q^{2^i})^{2^i} \right)^t \\
&=& 
\left(  \prod_{i\geq 0} \left( 1+2\sum_{k=1}^\infty (q^{2^i})^{k^2} \right)^{2^i} \right)^t \\
&=& 
\left( 1+2\sum_{k=1}^\infty q^{k^2} \right)^t\left(  \prod_{i\geq 1} \left( 1+2\sum_{k=1}^\infty (q^{2^i})^{k^2} \right)^{2^i} \right)^t \\
&\equiv & 
\left( 1+2\sum_{k=1}^\infty q^{k^2} \right)^t \pmod{4}\\
&=& 
\sum_{j=0}^t \binom{t}{j} 2^j \left( \sum_{k\geq 1}q^{k^2} \right)^j \\
&\equiv &
\sum_{j=0}^1 \binom{t}{j} 2^j \left( \sum_{k\geq 1}q^{k^2} \right)^j \pmod{4}\\
&=& 
1+2t\sum_{k\geq 1}q^{k^2}.
\end{eqnarray*}
The result follows given that $t$ is odd.  
\end{proof}

\vskip .25in

\item{}  Lastly, it should be clear that related results will also hold for even values of $t$; indeed, one can imagine even stronger divisibility properties in the case of $t$ even because of the divisibility by powers of 2 of the binomial coefficients that will arise in the expansion of the generating function when $t$ is even. It may be of interest to identify the corresponding divisibilities that can be guaranteed for all even values of $t$ (as a cousin to Theorem \ref{mainthm}).  We leave this as an exercise for the interested reader.  

\end{itemize}

%%%%%%%%%%%%%%%%%%%%%%%%%%%%%%%%%%%%%%%%%%%%
%%%%%%%%%%%%%%%%%%%%%%%%%%%%%%%%%%%%%%%%%%%%
%%%%%%%%%%%%%%%%%%%%%%%%%%%%%%%%%%%%%%%%%%%%
%%%%%%%%%%%%%%%%%%%%%%%%%%%%%%%%%%%%%%%%%%%%
%%%%%%%%%%%%%%%%%%%%%%%%%%%%%%%%%%%%%%%%%%%%
%%%%%%%%%%%%%%%%%%%%%%%%%%%%%%%%%%%%%%%%%%%%
%%%%%%%%%%%%%%%%%%%%%%%%%%%%%%%%%%%%%%%%%%%%
%%%%%%%%%%%%%%%%%%%%%%%%%%%%%%%%%%%%%%%%%%%%


\begin{thebibliography}{00}

\bibitem{CHSZ}
W. Y. C. Chen, Q-H. Hou, L. H. Sun, L. Zhang, 
Ramanujan-type congruences for overpartitions modulo 16, {\it Ramanujan J.} {\bf 40}, no. 2 (2016), 311--322.

\bibitem{CL}  
S. Corteel and J. Lovejoy, Overpartitions, {\it Trans. Amer. Math. Soc.} {\bf 356}, no. 4 (2004), 1623--1634.  

\bibitem{H} M. D. Hirschhorn, {The power of $q$, a personal journey}, Developments in Mathematics, v. 49, Springer (2017).

\bibitem{HS_INT}
M. D. Hirschhorn and J. A. Sellers, An infinite family of overpartitions congruences modulo 12, {\it INTEGERS} {\bf 5}, no. 1 (2005), paper \#A20.  

\bibitem{HS_JCMCC}
M. D. Hirschhorn and J. A. Sellers, 
 Arithmetic relations for overpartitions, {\it J. Combin. Math. Combin. Comput.} {\bf 53} (2005), 65--73.

\bibitem{KSV} 
D. M. Keister, J. A. Sellers,  and R. G. Vary, Some arithmetic properties of overpartition $k$--tuples, {\it INTEGERS} {\bf 9} (2009), paper \#A17.

\bibitem{BK_INT}
B. Kim,  The overpartition function modulo 128, {\it INTEGERS}  {\bf 8} (2008), paper \#A38.

\bibitem{BK_DM} 
B. Kim, A short note on the overpartition function, {\it Discrete Math.} {\bf 309}, no. 8 (2009), 2528--2532.

\bibitem{BK_DM2}
B. Kim, Overpartition pairs modulo powers of 2, {\bf Discrete Math.} {\bf 311}, no. 10--11 (2011), 835--840.

\bibitem{KM}
K. Mahlburg,  The overpartition function modulo small powers of 2, {\it 
Discrete Math.} {\bf 286}, no. 3 (2004), 263--267.

\bibitem{NN}
S. S. Nayaka and M. S. M Naika, Congruences modulo powers of 2 for $t$--colored overpartitions, {\it Bol. Soc. Mat. Mex. (3)} {\bf 28}, no. 3 (2022), paper \#66.  

\bibitem{SR} 
C.-S. Radu, An algorithmic approach to Ramanujan-Kolberg identities, {\it J. Symbolic Comput.} {\bf 68} (2015), 225--253.

\bibitem{MS} 
M. P. Saikia, Some missed congruences modulo powers of 2 for $t$--colored overpartitions, {\it Bol. Soc. Mat. Mex. (3)} {\bf 29}, no. 1 (2023), paper \#15.  

\bibitem{NS} 
N. A. Smoot, On the computation of identities relating partition numbers in arithmetic progressions with eta quotients: an implementation of Radu's algorithm, 
{\it J. Symbolic Comput.} {\bf 104} (2021), 276--311.

\bibitem{LW} 
L. Wang, Arithmetic properties of overpartition triples, {\it Acta Math. Sin. (Engl. Ser.)} {\bf 33}, no. 1 (2017), 37--50.


\end{thebibliography}
\end{document}